\documentclass[11pt]{amsart}

\usepackage{amsmath}
\usepackage{amssymb}
\usepackage{amscd}
\usepackage{color}

\begin{document}

\title{Reversibility of rings with respect to the Zhou Radical}

\author {Tugce Pekacar Calci}
\address{Tugce Pekacar Calci, Department of Mathematics, Ankara University,  Turkey}
\email{tcalci@ankara.edu.tr}

\author {Serhat Emirhan Soycan}
\address{
Ankara University Graduate School of National and Applied Sciences, Ankara, Turkey}
\email{<sesoycan@ankara.edu.tr>}

\newtheorem{thm}{Theorem}[section]
\newtheorem{lem}[thm]{Lemma}
\newtheorem{prop}[thm]{Proposition}
\newtheorem{cor}[thm]{Corollary}
\newtheorem{exs}[thm]{Examples}
\newtheorem{defn}[thm]{Definition}
\newtheorem{nota}{Notation}
\newtheorem{rem}[thm]{Remark}
\newtheorem{ex}[thm]{Example}
\newtheorem{que}[thm]{Question}

\begin{abstract} Let $R$ be a ring with identity and $\delta(R)$ denote the Zhou
radical of $R$. A ring $R$ is called {\it $\delta$-reversible}
 if for any $a$, $b \in R$,  $ab = 0$ implies $ba \in \delta(R)$. In this paper,
 we give some properties of $\delta$-reversible rings. We examine some extensions of $\delta$-reversible rings. \vspace{2mm}


\end{abstract}

\maketitle

\section{Introduction} Throughout this paper all rings are associative with identity
unless otherwise stated. 

In module theory, $\delta$ submodule has an important role. Let $M$ be a right module and $N$ a submodule of $M$. Zhou introduced the notion of $\delta$-small
submodule in \cite{Zhou}. He called $N$ \textit{$\delta$-small} in $M$ if whenever $M = N + L$ and $M/L$ is singular, then $M = L$. The sum of $\delta$-small submodules is denoted by $\delta(M)$. Considering the
ring $R$ as a right $R$-module over itself, the ideal $\delta(R)$ is introduced as a sum of $\delta$-small
right ideals of $R$. For the reason that Zhou introduced the delta submodule $\delta(M)$ of a
module $M$, $\delta(M)$ is called \textit{Zhou radical} of $M$ in \cite{Burcu}.

According to
Cohn \cite{Co}, a ring $R$ is said to be $reversible$ if for any
$a$, $b\in R$, $ab = 0$ implies $ba = 0$. In \cite{Mbc}, the authors studied $J$-reversible rings. A ring $R$ is said to be \textit{$J$-reversible}, if for $a,b \in R$, $ab=0$ implies $ba\in J(R)$. $\delta(R)$ is the intersection of all essential maximal right ideals of
$R$ and $J(R)$ is the intersection of all maximal right ideals. So clearly, $J(R) \subseteq \delta(R)$. With this motivation we define $\delta$-reversible rings. A ring $R$ is said to be \textit{$\delta$-reversible}, if for $a,b \in R$, $ab=0$ implies $ba\in \delta(R)$. It is clear that every $J$-reversible ring is $\delta$-reversible. But the converse is not true in general. If $S_r\subseteq J(R)$ for a ring $R$, then $R$ is
$\delta$-reversible if and only if it is $J$-reversible where $S_r$ denotes the right socle of a ring. Recall that a ring $R$ is called $\delta$-clean if for every $x\in R$, there exist an idempotent $e$ and $d\in \delta(R)$ such that $a=e+d$, \cite{ozcangurgun}. In this direction we prove that every $\delta$-clean ring is $\delta$-reversible. We investigate the properties of $\delta$-reversible rings. We give an example to illustrate that $\delta$-reversible rings are not abelian (i.e., rings in which idempotents are central), but if R is $\delta$-reversible and every idempotent lifts modulo $\delta(R)$, then $R/\delta(R)$ is abelian. Section 3 deals with some extensions of rings. It is proved $R$ is $\delta$-reversible if and only if upper triangular matrix
ring over $R$ is $\delta$-reversible. Also we prove that $n \times n$ matrix rings over $R$ are need not to be $\delta$-reversible for a ring $R$.

In what follows, $\mathbb{Z}$ denote the ring of integers and the ring of integers modulo $n$ is
denoted by $\mathbb{Z}_n$. We write
$M_n(R)$ for the ring of all $n\times n$ matrices and $T_n(R)$ for
the ring of all $n\times n$ upper triangular matrices over $R$. We denote
$S_r =$ Soc$(R_R)$ for the right socle of a ring $R$. Also we write $R[[x]]$, $J(R)$ and $U(R)$ for the power series ring over a ring $R$, the Jacobson radical of $R$ and the set of units of $R$, respectively. Moreover $l_R(a)$ and $r_R(a)$ denote the set of left annihilators and right annihilators of $a$ in $R$, respectively.

\section{$\delta$-Reversible Rings}

In this section we introduce a class of rings, called $\delta$-reversible rings, which
is a generalization of $J$-reversible rings. We investigate which properties of
$J$-reversible rings hold for this general setting. We give relations between $\delta$-reversible rings and some related rings. We begin with the equivalent conditions for the Zhou radical which is
proved in \cite{Zhou}.

\begin{thm}\cite[Theorem 1.6]{Zhou}
    Let $R$ be a ring. The following sets are equal to $\delta(R)$.
    \begin{enumerate}
        \item $R_1=$ the intersection of all essential maximal right ideals of $R$.
        \item $R_2=$ the unique largest $\delta$-small right ideal of $R$.
        \item $R_3=\{x\in R:xR+K_R=R$ implies $K_R$ is a direct summand of $R_R\}$.
        \item $R_4=\bigcap\{ \text{ideals } P \text{ of } R:~R/P \text{ has a faithful singular simple module}\}$.
        \item $R_5=\{x\in R : \text{for all } y\in R \text{ there exists a semisimple right ideal } Y \text{ of } R \text{ such that } \\
        (1+xy)R\oplus Y=R_R\} $
        
    \end{enumerate}
\end{thm}

\begin{rem}\rm(\cite[Corollary 1.7]{Zhou}\rm)
\begin{enumerate}
\item The ideal $\delta(R)$ of a ring $R$ is also characterized as $J(R/S_r) = \delta(R)/S_r$
where $S_r$ denotes the right socle of the ring $R$, that is, $S_r$ is the sum of minimal right
ideals of $R$.   
\item $\delta(R)$ is the intersection of all essential maximal right ideals of
$R$ and $J(R)$ is the intersection of all maximal right ideals. So clearly, $J(R) \subseteq \delta(R)$.
Furthermore $R=\delta(R)$ if and only if $R$ is semisimple.
\end{enumerate}
\end{rem}

The following Lemma helps us throughout the paper.

\begin{lem}\label{semiprime}
$\delta(R)$ is a semiprime ideal of a ring $R$.
\end{lem}

\begin{proof}
Assume that $aRa\subseteq \delta (R)$ and $a\notin \delta(R)$.
Then there exists a essential maximal right ideal $M_i$ of $R$
such that $a\notin M_i$. Hence $aR+M_i=R$. Thus $1=m_i+ar$ where
$m_i\in M_i$, $r\in R$. We have $a=m_ia+ara$ which is a
contradiction. So $a\in \delta(R)$.
\end{proof}

We now we represent our main definition.

\begin{defn} {\rm A ring $R$ is said to be $\delta$-reversible ring, if
$ab=0$ for any $a,b\in R$ implies $ba\in \delta(R)$.}
\end{defn}

All reduced, symmetric and reversible rings are $\delta$-reversible. It is clear that every $J$-reversible ring is $\delta$-reversible,
as $J(R)\subseteq \delta(R)$. The following example shows that the
converse does not hold in general.

\begin{ex}\label{tersornek}
{\rm Consider the ring $M_2(\mathbb Z_3)$. By \cite[Lemma 2.5]{Burcu}, $\delta (M_2(\mathbb
Z_3))=M_2(\delta (\mathbb Z_3))$. Also $J (M_2(\mathbb Z_3))=M_2(J (\mathbb
Z_3))$. As $\mathbb Z_3$ is semisimple, $\delta(\mathbb Z_3)=\mathbb Z_3$.
Then $\delta(M_2(\mathbb Z_3))=M_2(\mathbb Z_3)$. So it is clear that
$M_2(\mathbb Z_3)$ is $\delta$-reversible. Also we have $J(\mathbb
Z_3)=\{\overline 0\}$ and so
$J(M_2(\mathbb Z_3))=\begin{bmatrix}
  \overline 0 & \overline 0 \\
  \overline 0 & \overline 0 \\
\end{bmatrix}$. Choose $A=\begin{bmatrix}
  \overline 1 & \overline 2 \\
  \overline 0 & \overline 0\\
\end{bmatrix}$ and $B=\begin{bmatrix}
  \overline 2 & \overline 0 \\
  \overline2 & \overline 0 \\
\end{bmatrix}$ in $M_2(\mathbb Z_3)$. Although $AB=\begin{bmatrix}
  \overline 0 & \overline 0 \\
  \overline 0 & \overline 0 \\
\end{bmatrix}$, $BA=\begin{bmatrix}
  \overline 2 & \overline 1 \\
  \overline 2 & \overline 1 \\
\end{bmatrix}\notin J(M_2(\mathbb Z_3))$. Therefore $M_2(\mathbb Z_3)$ is not
a $J$-reversible ring.}
\end{ex}

The following proposition exhibits $\delta$-reversible rings are
$J$-reversible under which conditions.

\begin{prop}
Let $R$ be a ring and $S_r\subseteq J(R)$. Then $R$ is
$\delta$-reversible if and only if it is $J$-reversible.
\end{prop}

\begin{proof}
Assume that $R$ is $\delta$-reversible and $S_r\subseteq J(R)$.
Then $J(R)/S_r=J(R/S_r)=\delta(R)/S_r$ by \cite
[Corollary 1.7]{Zhou}. Therefore $\delta(R)=J(R)$. Hence $R$ is
$J$-reversible. The converse is clear.
\end{proof}

\begin{thm}
    Let $R/S_r$ be a $J$-reversible ring, then $R$ is $\delta$-reversible.
\end{thm}

\begin{proof}
    Let $ab=0$. Then $\overline{ab}=0$. It means $\overline{ba}\in J(R/S_r)=\delta(R)/S_r$. Since $S_r\subseteq \delta(R)$, $ba\in \delta(R)$.
\end{proof}

Recall that a ring $R$ is called \textit{abelian} if all idempotent elements of it are central. 

\begin{prop}\label{abel1}
Let $R$ be a ring and every idempotent of $R$ lift modulo
$\delta(R)$. Then $R$ is $\delta$-reversible, $R/\delta(R)$ is an
Abelian ring.
\end{prop}

\begin{proof}
Let $ \overline{f}=\overline{f}^2\in R/\delta(R)$. By assumption,
there exists $e^2=e\in R$ such that $e-f\in \delta(R)$. Then $e(1-e)=0$ and so $e(1-e)x=0$ for
all $x\in R$. As $R$ is $\delta$-reversible, $ex-exe\in
\delta(R)$. Via similar argument, we have $xe-exe\in \delta(R)$.
Therefore $xe-ex\in \delta(R)$. This implies that
$ex+\delta(R)=xe+\delta(R)$.
\end{proof}

\begin{cor}\label{idempoentli}
Let $R$ be a $\delta$-reversible ring. Then for every $e^2=e\in
R$, $eR(1-e)+(1-e)Re\subseteq \delta(R)$.
\end{cor}

Corollary \ref{idempoentli} is very useful to determine if a ring
is not $\delta$-reversible.

\begin{ex}\label{matrixring}
Let $R=M_2(\mathbb Z)$. Then $\delta(M_2(\mathbb Z))=
\begin{bmatrix}
  0 & 0 \\
  0 & 0 \\
\end{bmatrix}$ since $\delta(\mathbb Z)=0$. For idempotent $e=
\begin{bmatrix}
  1 & 0 \\
  0 & 0 \\
\end{bmatrix}$, $eR(1-e)+(1-e)Re=\begin{bmatrix}
  0 & \mathbb Z \\
  \mathbb Z & 0 \\
\end{bmatrix}$. Therefore $R$ is not $\delta$-reversible, by Corollary
\ref{idempoentli}.
\end{ex}

Let $R$ be a ring and $I$ an ideal of $R$. Recall that $I$ is
semiprime, if $aRa\subseteq I$ implies $a\in I$, for some $a\in
R$.

The next result is important to decide if a ring is
$\delta$-reversible, or not. For the proof of next Lemma, you
should look up to \cite{Mbc}.

\begin{lem}
The following are equivalent for a ring $R$.
\begin{enumerate}
\item $R$ is $\delta$-reversible.
\item For any $r\in R$, if $r^2=0$, then $r\in \delta(R)$.
\item For any $r\in R$, if $r^2=0$, then $rm-mr\in \delta(R)$ for all
$m\in R$.
\end{enumerate}
\end{lem}

Now we give a different characterization for $\delta$-reversible
rings.

\begin{thm}
The following are equivalent for a ring $R$.
\begin{enumerate}
\item $R$ is $\delta$-reversible.
\item For all $a\in R$, $a(l_R(a))\subseteq \delta(R)$.
\item For all $a\in R$, $(r_R(a))a\subseteq \delta(R)$.
\end{enumerate}
\end{thm}

\begin{proof}
(1) $\Rightarrow$ (2) Let $a\in R$ and $x\in l_R(a)$. So $xa=0$. As
$R$ is $\delta$-reversible, $ax\in \delta(R)$. Thus we have
$al_R(a)\subseteq \delta(R)$.\\
(2) $\Rightarrow$ (3) Let $x\in r_R(a)$. Hence $a\in l_R(x)$. By (2) $xa\in \delta(R).$ So $r_R(a)a\subseteq \delta(R)$ \\
(3) $\Rightarrow$ (1) For $a,b\in R$ assume that $ab=0$. Then we have $ba\in \delta(R)$ by(3).
\end{proof}

Let $R$ be a ring and $I$ an ideal of $R$. Then $I$ is called $\delta$-reversible, if for any $a,b\in I$, $ab=0$ implies $ba\in \delta(I)$. Then we have the following result.

\begin{prop}
Let $R$ be a ring and $I$ an ideal of $R$. If $R$ is a
$\delta$-reversible ring, then $I$ is $\delta$-reversible.
\end{prop}

\begin{proof}
Suppose that $ab=0$ for $a,b\in I$. Then $ba\in \delta(R)$, as $R$
is a $\delta$-reversible ring. Since $\delta(I)=I\cap
\delta(R)$, $ba\in \delta(I)$.
\end{proof}

\begin{prop}
Let $\{R_i\}_{i\in \mathcal I}$ be a family of rings for an
indexed set $\mathcal{I}$. Then $\prod_{i\in \mathcal I}R_i$ is
$\delta$-reversible if and only if so is $R_i$, for any $i\in
\mathcal I$.
\end{prop}

\begin{proof}
Assume that $R_i$ is $\delta$-reversible for any $i\in \mathcal I$
and $(x_i)(y_i)=(0)$, for $(x_i),(y_i)\in \prod_{i\in \mathcal
I}R_i$. So $x_iy_i=0$ for every $i\in \mathcal I$. By hypothesis,
we have $y_ix_i\in \delta(R)$,for all $i\in \mathcal I$. Since
$\delta(\prod_{i\in \mathcal I}R_i)=\prod_{i\in \mathcal
I}(\delta(R_i))$, $(y_i)(x_i)\in \delta(\prod_{i\in \mathcal
I}R_i)$, as asserted. The converse is clear.
\end{proof}

\begin{prop}\label{ere}
A ring $R$ is $\delta$-reversible if and only if $eRe$ is
$\delta$-reversible for every idempotent $e\in R$.
\end{prop}

\begin{proof}
Let $a,b\in R$. Assume that $e^2=e\in R$ and $(eae)(ebe)=0$. Then $(ebe)(eae)\in
\delta(R)$ since $R$ is a $\delta$-reversible ring. Therefore
$(ebe)(eae)\in e\delta(R)e=\delta(eRe)$, as desired. The
sufficiency of proof is clear for $e=1$.
\end{proof}

Let  $\delta^{\#}(R)$ denote the subset $\{ x\in R ~|~ \exists ~n\in
{\mathbb N}~\mbox{such that}~x^n\in \delta(R)\}$ of $R$. It is obvious that $\delta(R)\subseteq\delta^{\#}(R)$, but the converse doesn't hold for general. Next example explains that.
\begin{ex}
    It is known that $\delta(M_2(\mathbb Z))=\left\{ \left[
\begin{array}{cc}
0&0\\
0&0
\end{array}
\right]\right\}$. Therefore\\
$A=\left[
\begin{array}{cc}
0&1\\
0&0
\end{array}
\right]\notin \delta(R)$, but $A^2=0\in \delta(R)$, so $A\in \delta^{\#}(R)$.
\end{ex} 
\begin{cor}
    If $R$ is a local ring, then $\delta^{\#}(R)=\delta(R)$.
\end{cor}

\begin{proof}
It is clear that $\delta(R)\subseteq\delta^{\#}(R)$. For the converse, let $x\in \delta^{\#}(R)$. If $x\in U(R)$, then $x^n\in \delta(R)$. But since $x^n$ is unit, $\delta(R)=R$. If $x\notin U(R)$, then $x\in J(R)\subseteq \delta(R)$. So proof is completed.
\end{proof}

Now we have the following.

\begin{prop}
    Let $R$ be a ring. If $\delta^{\#}(R)=\delta(R)$, then $R$ is $\delta$-reversible.
\end{prop}

\begin{proof}
    Assume that $\delta^{\#}(R)=\delta(R)$ and $ab=0$ for any $a,b\in R$. Since $(ba)^2=0$, $ba\in \delta^{\#}(R)=\delta(R)$ and this completes the proof.
\end{proof}

\begin{cor}
    Let $R$ be a ring. If $R/\delta(R)$ is reduced, then $R$ is $\delta$-reversible.
\end{cor}

\begin{proof}
    Let $R/\delta(R)$ is reduced and $x\in \delta^{\#}(R)$. Then for an $n\in \mathbb{N}$, $x^n\in \delta(R)$. So $\overline{x^n}=\overline{0}$ and since $R/\delta(R)$  reduced, $\overline{x}=\overline{0}$ and this implies $x\in \delta(R)$, as asserted.
\end{proof}

\begin{defn}
Let $R$ be a ring, $f(x)=a_0+a_1x+\cdots+a_nx^n\in R[x]$ and
$g(x)=b_0+b_1x+\cdots+b_mx^m\in R[x]$. Then $R$ is
$\delta$-Armendariz, if $f(x)g(x)=0$, then $a_ib_j\in \delta(R)$
for any $0\leq i\leq n$, $0\leq j\leq m$. Furthermore, $R$ is
called $\delta$-linear Armendariz, if for $f(x)=a_0+a_1x$ and
$g(x)=b_0+b_1x$ being $f(x)g(x)=0$ implies that $a_ib_j\in
\delta(R)$ for any $0\leq i,j \leq 1$.
\end{defn}

\begin{prop}
Let $R$ be a ring. If $R$ is $\delta$-reversible, then it is
$\delta$-linear Armendariz.
\end{prop}

\begin{proof}
Assume that $f(x)g(x)=0$ for $f(x)=a_0+a_1x\in R[x]$ and
$g(x)=b_0+b_1x\in R[x]$. Then we have
$a_0b_0+(a_0b_1+a_1b_0)x+a_1b_1x^2=0$. So $a_0b_0=0$ and
$a_1b_1=0$. Thus
$b_0a_0\in \delta(R)$ and $b_1a_1\in \delta(R)$ since $R$ is
$\delta$-reversible ring. It is easy to show that for every $r\in R$, $a_0b_0r=0$. Then $b_0 r a_0 \in \delta (R)$. It is clear that
$(a_0b_1+a_1b_0)ra_0=a_0b_1ra_0+a_1b_0ra_0=0$. As $b_0ra_0\in
\delta(R)$, we have $a_0b_1ra_0\in \delta(R)$. Thus $a_0b_1ra_0b_1 \in \delta(R)$ By Lemma
\ref{semiprime}, $a_0b_1\in \delta(R)$. Via similar argument, one can show that $a_1b_0\in
\delta(R)$. Hence $R$ is $\delta$-linear Armendariz.
\end{proof}

Recall that a ring $R$ is called $\delta$-clean if for every $x\in R$, there exist an idempotent $e$ and $d\in \delta(R)$ such that $a=e+d$, \cite{ozcangurgun}. In this direction we have the following.

\begin{thm}\label{deltaclean}
    Every $\delta$-clean ring is $\delta$-reversible.
\end{thm}

\begin{proof}
    Let $a,b\in R$ and $ab=0$. Then $(ba)^2=0$ and since $R$ is $\delta$-clean there exist $e^2=e\in R$ and $d\in \delta(R)$ such that $ba=e+d$. So $(e+d)^2=e+ed+de+d^2=0$ it implies that $e+d+ed+de+d^2=ba+ed+de+d^2=d$. So $ba=d-ed-de-d^2\in \delta(R)$ as asserted.
\end{proof}

\begin{lem}\label{deltaquasi}
Let $R$ be a $\delta$-quasipolar ring. If $a\in nil(R)$, then
$\delta$-spectral idempotent of $a$ is contained in  $\delta(R)$.
\end{lem}

\begin{proof}
Assume that $R$ is $\delta$-quasipolar ring and $a^n=0$ for $n\in
\mathbb Z^+$. Then there exists $p^2=p\in comm^2(a)$ such that
$a+p\in \delta(R)$. We have $(a+p)^n=a^n+\left(%
\begin{array}{c}
  n \\
  1 \\
\end{array}%
\right)a^{n-1}p+\left(%
\begin{array}{c}
  n \\
  2 \\
\end{array}%
\right)a^{n-2}p+\cdots +\left(%
\begin{array}{c}
  n \\
  n-2 \\
\end{array}%
\right)a^2p+\left(%
\begin{array}{c}
  n \\
  n-1 \\
\end{array}%
\right)ap+p\in \delta(R)$, as $comm^2(a)\subseteq comm(a)$. Then
$a^{n-1}(a+p)^n=a^{n-1}p\in \delta(R)$. So $(a+p)^n-a^{n-1}p=\left(%
\begin{array}{c}
  n \\
  2 \\
\end{array}%
\right)a^{n-2}p+\cdots+p\in \delta(R)$. If the last equation is
multiplied by $a^{n-3}$, we conclude that $n a^{n-1}p+a^{n-2}p\in
\delta(R)$. Then $n a^{n-1}p+a^{n-2}p-n a^{n-1}p=a^{n-2}p\in
\delta(R)$, as $a^{n-1}p\in \delta(R)$. Hence $\left(%
\begin{array}{c}
  n \\
  2 \\
\end{array}%
\right)a^{n-2}p+\cdots+p-\left(%
\begin{array}{c}
  n \\
  2 \\
\end{array}%
\right)a^{n-2}p=\left(%
\begin{array}{c}
  n \\
  3 \\
\end{array}%
\right)a^{n-3}p+\cdots+\left(%
\begin{array}{c}
  n \\
  n-1 \\
\end{array}%
\right)ap+p\in \delta(R)$. If we continue similarly, we have $p\in
\delta(R)$.
\end{proof}

\begin{cor}\label{quasi-revers}
Every $\delta$-quasipolar ring is $\delta$-reversible.
\end{cor}

\begin{proof}
It is a direct consequence of Theorem \ref{deltaclean}. To show the synergy of ideas, here is another proof. Let $R$ be a $\delta$-quasipolar ring and $ab=0$ for $a,b\in R$.
Then $(ba)^2=0$. There exists $p^2=p\in comm^2(ba)$ such that
$ba+p\in \delta(R)$, by $\delta$-quasipolarity of $R$. By Lemma
\ref{deltaquasi}, $p\in \delta(R)$. Thus $ba\in \delta(R)$.
\end{proof}

The following example shows that the converse of Corollary \ref{quasi-revers} is not satisfied in general.

\begin{ex}
Let $R=\{(q_1,q_2,\ldots,q_n,a,a,\ldots):~n\geq 1,~q_i\in \mathbb Q,
a\in \mathbb Z_2\}$. $R$ is a $\delta$-reversible ring since it is
commutative. But $R$ is not $\delta$-quasipolar by \cite[Example
5]{Tugce}.
\end{ex}

\section{Some Extensions of $\delta$-Reversible Rings}

In this section, we study ring extensions in terms of $\delta$-reversibility.

\begin{lem}\cite[Proposition 3.15]{ozcangurgun}\label{polinomial}
Let $R$ be a ring. Then $\delta(R[[x]])\subseteq \delta(R)+\langle x\rangle$.
\end{lem}

\begin{prop}
Let $R$ be a ring. If $R[[x]]$ is a $\delta$-reversible ring, Then $R$ is a $\delta$-reversible ring.
\end{prop}

\begin{proof}
The proof is clear via Lemma \ref{polinomial}. 
\end{proof}

Recall that a \textit{Morita context} is a 4-tuple  $
\begin{bmatrix}
  A & M \\
  N & B \\
\end{bmatrix}$ where $A,B$ are rings, $_AM_B$ and $_BN_A$ are bimodules and there exist context products $M\times N \rightarrow A$ and $N\times M \rightarrow B$ written multiplicatively as $(w,z)\mapsto wz$ and $(z,w)\mapsto zw$ such that $
\begin{bmatrix}
  A & M \\
  N & B \\
\end{bmatrix}$ is an associative ring with the obvious matrix operations. 

Morita contexts were introduced by Morita in \cite{Mor}. A Morita context $
\begin{bmatrix}
  A & M \\
  N & B \\
\end{bmatrix}$ is called \textit{trivial} if the context products are trivial, i.e. $MN=0$ and $NM=0$. This is also called \textit{null context}.

\begin{lem}\label{Morita}
    Let $R=
\begin{bmatrix}
  A & M \\
  N & B \\
\end{bmatrix}$ is a trivial Morita context. Then $\delta(R)\subseteq 
\begin{bmatrix}
  \delta(A) & M \\
  N & \delta(B) \\
\end{bmatrix}.$
\end{lem}

\begin{proof}
Let $r=
\begin{bmatrix}
  a & m \\
  n & b \\
\end{bmatrix} \in \delta(R)$. To see that $a\in \delta(A)$, let $aA+I=A$ where $I$ is a right ideal of $A$. It is enough to show that $I$ is a direct summand of $A$. Since $rR+ 
\begin{bmatrix}
  I & M \\
  N & B \\
\end{bmatrix}=R$ and $r\in \delta(R)$, $
\begin{bmatrix}
  I & M \\
  N & B \\
\end{bmatrix}$ is a direct summand of $R$. Hence there exists $\Big(
\begin{bmatrix}
  f & m \\
  n & g \\
\end{bmatrix}\Big)^2=
\begin{bmatrix}
  f & m \\
  n & g \\
\end{bmatrix} \in R$ such that $
\begin{bmatrix}
  I & M \\
  N & B \\
\end{bmatrix}=
\begin{bmatrix}
  f & m \\
  n & g \\
\end{bmatrix}R$. Thus we have $f^2=f \in A$ and $I=fA$ which shows that $I$ is a direct summand of $A$. So $a\in \delta(A)$. Similarly, it can be shown that $b\in \delta(B)$. 
\end{proof}

\begin{thm} 
   Let $R=
\begin{bmatrix}
  A & M \\
  N & B \\
\end{bmatrix}$ is trivial Morita context. If $R$ is $\delta$-reversible, then $A$ and $B$ are $\delta$-reversible.
\end{thm}

\begin{proof}
Assume that $R$ is $\delta$-reversible. Then by Lemma \ref{ere}, $A$ and $B$ are $\delta$-reversible. 
\end{proof}

 Let $S$ and $T$ be any rings, $M$ an
$S$-$T$-bimodule and $R$ the formal triangular matrix ring $
\begin{bmatrix}
S&M\\
0&T\\
\end{bmatrix}
$. It is proved in \cite[Lemma 5.1]{ozcangurgun} that $\delta(R)\subseteq
\begin{bmatrix}
\delta(S)&M\\
0&\delta(T)\\
\end{bmatrix}
$.

\begin{cor}\label{moritacont}
Let $R=
\begin{bmatrix}
S&M\\
0&T\\

\end{bmatrix}
$. If $R$ is $\delta$-reversible, then $S$ and $T$ are $\delta$-reversible.
\end{cor}

\begin{cor}\label{tnr}
    Let $R$ be a ring. If $T_n(R)$ is $\delta$-reversible, then $R$ is $\delta$-reversible.
\end{cor}

\begin{rem}
Let $R$ be a ring. Even if $R$ is $\delta$-reversible, $M_n(R)$ does not have to be $\delta$-reversible in general. We may assume $n=2$ and $R=\mathbb Z$. It is clear that $\mathbb Z$ is $\delta$-reversible. Nevertheless $M_2(\mathbb Z)$ is not a $\delta$-reversible ring by Example \ref{matrixring}.
\end{rem}

\noindent{\bf The rings $H_{(s,t)}(R)$:} Let $R$ be a ring and  $s, t\in
C(R)\cap U(R)$. Set $$H_{(s,t)}(R) = \left
\{\begin{bmatrix}a&0&0\\c&d&e\\0&0&f
\end{bmatrix}\in M_3(R)\mid a, c, d, e, f\in R, a - d = sc, d - f = te\right \}.$$
Then $H_{(s,t)}(R)$ is a subring of $M_3(R)$. 
 Now we give an important result which is proved in \cite{Burcu}.

\begin{lem}\label{Hst}
    Let $R$ be a ring. Then $$\delta(H_{(s,t)}(R))=\left \{\begin{bmatrix}a&0&0\\c&d&e\\0&0&f
\end{bmatrix} \in M_2(R) : a,d,f \in \delta(R)\right \}.$$
\end{lem}

Now we have the following.

\begin{thm}\label{Hdelta}
    Let $R$ be a ring. Then $R$ is $\delta$-reversible if and only if $H_{(s,t)}(R)$ is $\delta$-reversible.
\end{thm}

\begin{proof}
     Let $R$ be a $\delta$-reversible ring. Assume that $A=\begin{bmatrix}a_1&0&0\\c_1&d_1&e_1\\0&0&f_1
\end{bmatrix}$, $B=\begin{bmatrix}a_2&0&0\\c_2&d_2&e_2\\0&0&f_2
\end{bmatrix} \in H_{(s,t)}(R)$ such that $AB=0$. With a basic computation we handle $a_1a_2=d_1d_2=f_1f_2=0$. Since $R$ is $\delta$-reversible, $a_2a_1,d_2d_1,f_2f_1 \in \delta(R)$. So $BA=\begin{bmatrix} a_2a_1&0&0\\c_2a_1+d_2c_1&d_2d_1&d_2e_1+e_2f_1\\0&0&f_2f_1
\end{bmatrix} \in \delta(H_{(s,t)}(R))$ by Lemma \ref{Hst}. The converse is clear.
\end{proof}

\noindent {\bf The rings $L_{(s,t)}(R)$:} Let $R$ be a ring and $s$, $t\in
C(R)\cap U(R)$. Set 
$$L_{(s,t)}(R) = \left
\{\begin{bmatrix}a&0&0\\sc&d&te\\0&0&f\end{bmatrix}\in M_3(R)\mid
a, c, d, e, f\in R\right \}.$$ \\Then $L_{(s,t)}(R)$ is a subring of $M_3(R)$. Now we give an important result which is proved in \cite{Burcu}

\begin{lem}\label{Lst}
    Let $R$ be a ring. Then $$\delta(L_{(s,t)}(R))=  \left
\{\begin{bmatrix}a&0&0\\sc&d&te\\0&0&f\end{bmatrix}\in M_3(R)\mid
a, d, f\in \delta(R) \right \}  $$
\end{lem}
Now we have the following.
\begin{thm}
    Let $R$ be a ring. Then $R$ is $\delta$-reversible if and only if $L_{(s,t)}(R)$ is $\delta$-reversible.
\end{thm}

\begin{proof}
    Proof is similar with Theorem \ref{Hdelta}.
\end{proof}

\noindent {\bf Generalized matrix rings:} Let $R$ be a ring and $s$ a central element of $R$. Then  $\begin{bmatrix}
R&R\\R&R\end{bmatrix}$ becomes a ring denoted by $K_s(R)$ with addition defined componentwise and multiplication defined in \cite{Kr} by
$$\begin{bmatrix} a_1&x_1\\y_1&b_1\end{bmatrix}\begin{bmatrix} a_2&x_2\\y_2&b_2\end{bmatrix} = \begin{bmatrix}a_1 a_2+sx_1y_2&a_1x_2+x_1b_2\\y_1a_2+b_1y_2&sy_1x_2+b_1b_2\end{bmatrix}.$$ In \cite{Kr}, $K_s(R)$ is called a {\it generalized matrix ring
over $R$}.
Now we give an important result.
\begin{lem}\label{K0R}
    $\delta(K_0(R))=\left
\{\begin{bmatrix}a&b\\c&d\end{bmatrix}\in K_0(R)\mid
a, d\in \delta(R)\right \}$.
\end{lem}

\begin{proof}
    $\left
\{\begin{bmatrix}a&b\\c&d\end{bmatrix}\in K_0(R)\mid
a, d\in \delta(R)\right \} \subseteq \delta(K_0(R))$ by \cite{Burcu}. $\delta(K_0(R)) \subseteq \left
\{\begin{bmatrix}a&b\\c&d\end{bmatrix}\in K_0(R)\mid
a, d\in \delta(R)\right \}$ can be shown similarly to the proof of Lemma \ref{Morita}.
\end{proof}

\begin{thm}
    Let $R$ be a ring. Then $R$ is $\delta$-reversible if and only if $K_0(R)$ is $\delta$-reversible.
\end{thm}

\begin{proof}
    Let $R$ be a $\delta$ reversible ring. Assume that $A=\begin{bmatrix}a&b\\c&d\end{bmatrix}$, $B=\begin{bmatrix}x&y\\z&t\end{bmatrix}$ and $AB=0$. So $ax=dt=0$. Since $R$ is $\delta$- reversible, $xa$, $td\in \delta(R)$. A basic computation shows that $BA\in \delta(R)$. By Lemma \ref{K0R}, converse is clear.
\end{proof}

\end{document}